\documentclass[12pt]{amsart}

\makeatletter
\@namedef{ver@rerunfilecheck.sty}{2026/07/05}
\makeatother
\newcommand{\RerunFileCheck}[4]{}

\let\degree\undefined

\usepackage{amsfonts, amsthm, amsmath, dsfont}
\usepackage{amssymb}
\usepackage{amscd}
\usepackage{lmodern}
\usepackage[latin2]{inputenc}
\usepackage{t1enc}
\usepackage[mathscr]{eucal}
\usepackage{mathrsfs}
\usepackage{mathtools} 

\usepackage{mathabx} 

\usepackage{graphicx}
\usepackage{graphics}
\usepackage{rotating}
\usepackage{pict2e}
\usepackage{epic}
\usepackage{framed}
\usepackage{booktabs}
\usepackage{makecell}
\usepackage{comment}
\usepackage{indentfirst}
\usepackage{enumerate}

\usepackage{tikz}
\usepackage{pgfplots}
\usetikzlibrary{arrows.meta, tikzmark}

\newcolumntype{V}{!{\vrule width 2pt}}

\usepackage[pagebackref]{hyperref}
\hypersetup{colorlinks=true}
\usepackage{bookmark} 

\usepackage{cite}
\usepackage{color}
\usepackage{refcount}

\allowdisplaybreaks
\numberwithin{equation}{section}

\theoremstyle{plain}
\newtheorem{theorem}{Theorem}[section]
\newtheorem{corollary}[theorem]{Corollary}
\newtheorem{proposition}[theorem]{Proposition}
\newtheorem{conjecture}[theorem]{Conjecture}
\newtheorem{remark}[theorem]{Remark}
\newtheorem{lemma}[theorem]{Lemma}
\newtheorem{definition}[theorem]{Definition}

\def\N{\mathbb{N}}
\def\fS{\mathfrak{S}}

\def\cM{\mathcal{M}}
\def\tail{\mathsf{tail}}
\def\des{\mathsf{des}}
\def\st{\mathrm{st}}
\def\rU{\mathrm{U}}
\def\rL{\mathrm{L}}
\def\rD{\mathrm{D}}
\def\cU{\mathcal{U}}
\def\cL{\mathcal{L}}
\def\cD{\mathcal{D}}
\def\rI{\mathrm{I}}
\def\cc{\mathrm{c}_2}
\def\nou{\mathrm{nu}}
\def\PK{\Psi_{\mathrm{K}}}
\def\Dyck{\mathrm{Dyck}}
\def\tu{\mathsf{tu}}
\def\rev{\mathrm{rev}}
\def\acb{\mathrm{\underline{13}2}}

\def\LRMax{\mathrm{LRMax}}

\def\multiset#1#2{\ensuremath{\left(\kern-.3em\left(\genfrac{}{}{0pt}{}{#1}{#2}\right)\kern-.3em\right)}}

\begin{document}
	
	\title[When arrow patterns meet classical patterns]{When arrow patterns meet classical patterns}
	
	\author[S. Fu]{Shishuo Fu$^{\ast}$}
	\address[S. Fu]{College of Mathematics and Statistics \& Center for Discrete Mathematics, Chongqing University, Chongqing 401331, China}
	\email{fsshuo@cqu.edu.cn}
	\thanks{$^{\ast}$Corresponding author: Shishuo Fu.}

	\author[Z. Yang]{Zhenghe Yang}
	\address[Z. Yang]{College of Mathematics and Statistics, Chongqing University, Chongqing 401331, China}
	\email{19332762303@163.com}

	\date{\today}

	\begin{abstract} 
	Seeking to bridge the structural divide between a permutation's cycle notation and its one-line notation, Berman and Tenner introduced a novel notion of permutation pattern known as the arrow pattern. Recently, Archer and Laudone initiated a systematic study of arrow pattern avoidance, leaving behind three intriguing conjectures. In this paper, we resolve all three conjectures. First, we enumerate all six subclasses of permutations that simultaneously avoid a classical pattern of length 3 and a fixed arrow pattern of length 3, thereby confirming the first two conjectures. Second, we settle the third conjecture (which involves a different arrow pattern) by providing two independent proofs. These proofs rely on a restriction of Biane's bijection to non-nesting involutions and Krattenthaler's bijection from $321$-avoiding permutations to Dyck paths, respectively.
	\end{abstract}

	
	\maketitle

	\section{Introduction}
	
	Let $\fS_n$ denote the set of permutations on $[n]:=\{1,2,\ldots,n\}$. Given two permutations $\sigma\in\fS_n$ and $\pi\in\fS_m$, we say that $\sigma$ \emph{contains} $\pi$ as a (classical) \emph{pattern}, if there exist $1\le i_1<i_2<\cdots<i_m\le n$ such that entries $\sigma_{i_1},\sigma_{i_2},\ldots,\sigma_{i_m}$ form a sequence that is \emph{order-isomorphic} to $\pi$. Otherwise $\sigma$ is said to \emph{avoid} $\pi$. We use $\fS_n(\pi)$ to denote the set of $n$-permutations that avoid the pattern $\pi$. The problem of enumerating various classes of pattern-avoiding permutations has spawned a stunning amount of work in enumerative combinatorics; see two book expositions~\cite{Bon2012,Kit11} for further information on this fast-developing field. 
	
	In a recent study on the so-called ``shallow'' permutations, Berman and Tenner~\cite{BT2022} introduced a new notion of permutation pattern called the \emph{arrow pattern}, whose definition we will recall in next section. Interest in shallow permutations stems from their role in understanding the Diaconis-Graham inequality~\cite{DG1977,Woo2022,CM2024,Lau2026}, which involves three fundamental permutation statistics: length, reflection length, and depth (or total displacement). As revealed by Berman and Tenner~\cite{BT2022}, the arrow pattern serves as a natural framework to simultaneously capture the structural information required by all three statistics.
	
	Archer and Laudone initiated in~\cite{AL2026} the enumeration of arrow pattern avoiding permutations. Towards the end of their paper, they paired arrow pattern avoidance with classical pattern avoidance and made the following three intriguing conjectures. Let $\pi$ be a classical pattern and $\alpha$ be an arrow pattern, then for every $n\in \N$, we denote by $a_n(\pi,\alpha):=|\fS_n(\pi,\alpha)|$ the number of $n$-permutations that avoid simultaneously two patterns $\pi$ and $\alpha$. Further notations and some preliminary results will be given in Section~\ref{sec:preliminaries}. 
	
	\begin{conjecture}[{\cite[Conjecture~7.1]{AL2026}}]\label{conj:Archer-Laudone}
		For $n\ge 2$, we have
		\begin{enumerate}
			\item $a_n(123,(12;1\to 3))=2^n-n$,
			\item $a_n(321,(12;1\to 3))=F_{2n-1}$,
			\item $a_n(321,(12;1\to 2))=M_n$,
		\end{enumerate}
		where $F_n$ is the $n$-th Fibonacci number~\cite[A000045]{OEIS} and $M_n$ is the $n$-th Motzkin number~\cite[A001006]{OEIS}.
	\end{conjecture}	

	 Motivated by this conjecture, in the current paper we carry out a complete enumeration of $\fS_n(\pi,\alpha)$, where $\pi$ ranges over all six classical patterns of length 3 and $\alpha=(12;1\to 3)$. The results are summarized in Table~\ref{tab:3-alpha all} and the proofs are given in Section~\ref{sec:alpha}. Recall that $C_n:=\frac{1}{n+1}\binom{2n}{n}$ is the $n$-th Catalan number~\cite[A000108]{OEIS}. In particular, this confirms items (1) and (2) from Conjecture~\ref{conj:Archer-Laudone}. The remaining part (3) is proved and refined/generalized in Section~\ref{sec:beta}, in two ways both of which are bijective in nature; see Theorems~\ref{thm:beta-Biane} and \ref{thm:beta-Kra}. We conclude the paper with a discussion on real-rootedness (Section~\ref{sec:real_rootedness}) and some remarks (Section~\ref{sec:concluding_remarks}) that hopefully could stimulate future research.

	\begin{table}[h]
    \renewcommand{\arraystretch}{1.2}
    \centering
	\begin{tabular}{|c|c|c|}
	\hline
	$\pi$ & $a_n(\pi,(12;1\to 3))$ & Ref. \\
	\hline
	123 & $2^n-n$ & Thm~\ref{thm:123-alpha}\\
	132 & $C_n$ & Thm~\ref{thm:132-alpha} \\
	213 & $F_{2n-1}$ & Thm~\ref{thm:213-231-q}\\
	231 & $F_{2n-1}$ & Thm~\ref{thm:213-231-q} \\
	312 & $C_n$ & Thm~\ref{thm:132-alpha}\\
	321 & $F_{2n-1}$ & Thm~\ref{thm:321-alpha} \\
	\hline
	\end{tabular}
	\medskip
	\caption{A complete enumeration of $\fS_n(\pi,(12;1\to 3))$ for all $\pi\in\fS_3$}\label{tab:3-alpha all}
	\end{table}

	\section{Preliminaries}\label{sec:preliminaries}

	The two most common ways to describe a permutation $\sigma\in\fS_n$ are \emph{one-line notation} and \emph{cycle notation}. For the one-line notation, we write $\sigma=\sigma_1\sigma_2\cdots\sigma_n$ with each $\sigma_i:=\sigma(i)$ representing the image of $i$ under the bijection $\sigma$. While for the cycle notation, $\sigma$ is written as a product of disjoint cycles, each of which records an orbit of $\sigma$ acting on $[n]$. For instance, the permutation $\sigma=5637421$ in its one-line notation can be rewritten using cycle notation as $\sigma=(1547)(26)(3)$. Because cycles can be cyclically shifted internally and multiplied in any order, the previous expression is equivalent to $(3)(62)(5471)$, as well as several other variations. A cycle notation is said to be \emph{standard} if 1) every cycle begins with its maximal element, and 2) the cycles are ordered increasingly from left to right by these maximal elements. Thus the standard cycle notation for $\sigma$ is $(3)(62)(7154)$.

	The definition of the arrow pattern as given by Berman and Tenner~\cite[Defn.~17]{BT2022} (see also \cite[Section 2]{AL2026}) requires a certain variant of a classical bijection called \emph{Foata's first fundamental transformation}~\cite[Chapter 10.2]{Lot1997}. We denote this variant by $\Phi: \fS_n\to \fS_n$ and recall its definition here for the sake of completeness. Given a permutation $\sigma\in\fS_n$ in its standard cycle notation, we remove all of the parentheses to obtain the one-line notation of a new permutation that we set to be the image $\Phi(\sigma)$. Returning to our running example $\sigma=5637421$, we see that $\Phi(\sigma)=3627154$.

	\begin{definition}[Arrow pattern]
	An \emph{arrow pattern} $\alpha=(\nu;H)$ of size $k$ consists of a string of positive integers $\nu=a_1\ldots a_m$ and a (possibly empty) collection of $h$ arrows $H=\{b_i\to c_i : 1\le i\le h\}$, so that the distinct integers appearing in either $\nu$ or $H$ form the set $[k]=\{1,2,\ldots,k\}$. A permutation $\sigma\in\fS_n$ is said to contain the arrow pattern $\alpha$ if the following conditions are satisfied. Otherwise we say $\sigma$ avoids $\alpha$.
	\begin{enumerate}
		\item There exists a subset $X=\{x_1,\ldots,x_k\}\subseteq [n]$ such that for some indices $t_1<\cdots <t_m$ we have $\sigma_{t_1}\cdots\sigma_{t_m}=x_{a_1}\cdots x_{a_m}$, and
		\item if $\hat{\sigma}=\Phi^{-1}(\sigma)$, then for each $1\le i\le h$, we have $\hat{\sigma}(x_{b_i})=x_{c_i}$.
	\end{enumerate}
	\end{definition}

	For example, consider the arrow patterns $\alpha=(12;1\to 3)$, $\beta=(231;1\to 4)$, and the permutation $\pi=3627154$ with $\hat{\pi}:=\Phi^{-1}(\pi)=5637421$. The permutation $\pi$ contains multiple occurrences of $12$ in the classical sense, for instance $36$, $25$, and $15$. Among these three occurrences, only $25$ is part of an occurrence of $\alpha$ since we can take $X=\{2,5,6\}$. In contrast, $\pi$ avoids the arrow pattern $\beta$, since for each occurrence of $231$ in $\pi$, no integer can be chosen to be the ``$4$'' in $\beta$. For example, $361$ is an occurrence of $231$ in $\pi$ in the classical sense, but $\hat{\pi}(1)=5<6$, preventing $361$ from being part of an occurrence of $\beta$ in $\pi$.

	Since all results displayed in Table~\ref{tab:3-alpha all} involve the arrow pattern $\alpha=(12;1\to 3)$, we prepare ourselves with the following characterization of the permutations that avoid $\alpha$. 

	\begin{lemma}\label{lem:char of alpha}
	A permutation $\sigma\in\fS_n$ avoids the arrow pattern $\alpha=(12;1\to 3)$ if and only if it is in one of the following two cases.
	\begin{enumerate}
		\item $\sigma_n=n$ and $\sigma':=\sigma_1\sigma_2\cdots\sigma_{n-1}\in\fS_{n-1}(\alpha)$;
		\item $\sigma=\tau_1\cdots\tau_i \eta_1'\cdots\eta_{n-k}'\tau_{i+2}\cdots\tau_{k}k$, for some $0\le i < k\le n-1$, such that $\tau:=\tau_1\cdots\tau_i k\tau_{i+2}\cdots\tau_k\in\fS_{k}(\alpha)$ and $\eta:=\eta_1\cdots\eta_{n-k}\in\fS_{n-k}(\alpha)$ with $\eta_j':=\eta_j+k$ for every $1\le j\le n-k$.
	\end{enumerate}
	\end{lemma}
	\begin{proof}
	This recursive characterization was utilized in the proof of \cite[Theorem 3.5]{AL2026} to show that $a_n(12;1\to 3)=S_{n-1}$, the $(n-1)$-st \emph{large Schr\"oder number}~\cite[A006318]{OEIS}. The proof is thus omitted.
	\end{proof}

	The two arrow patterns appearing in Conjecture~\ref{conj:Archer-Laudone} will be frequently mentioned in what follows, so we abbreviate them throughout as $\alpha:=(12;1\to 3)$ and $\beta:=(12;1\to 2)$. The notions of ``direct sum'', ``skew sum'', and ``standardization'' will be useful for our ensuing analysis and we recall them here. Let $\sigma\in\fS_n$ and $\tau\in\fS_m$ be two permutations, then the \emph{direct sum} and the \emph{skew sum} of $\sigma$ and $\tau$ are permutations of length $n+m$ defined respectively as
	$$\sigma\oplus\tau := \sigma_1\cdots\sigma_n(n+\tau_1)\cdots(n+\tau_m),$$
	and
	$$\sigma\ominus\tau := (m+\sigma_1)\cdots (m+\sigma_n)\tau_1\cdots\tau_m.$$
	For a word $w$ formed by letters from $\{\ell_1,\ldots,\ell_n\}_{<}$, its \emph{standardization}, denoted $\st(w)$, is defined as the permutation derived from $w$ by replacing $\ell_i$ with $i$, for every $1\le i\le n$. The \emph{intervals} in this paper contain only integers, so for $a,b\in\N$, $[a,b]$ represents the set of integers $\{a,a+1,\ldots,b-1,b\}$. Given a permutation $\sigma$, let 
	$$\LRMax(\sigma)=\{m_1=\sigma_1,m_2,\ldots,m_k=n\}$$ 
	denote the set of \emph{left-to-right maxima} of $\sigma$. I.e., if $m_i=\sigma_t$ for some $1\le i\le k$ and $1\le t\le n$, then we must have $\sigma_j<m_i$ for all $1\le j< t$.

\section{Arrow pattern \texorpdfstring{$(12;1\to 3)$}{alpha}}\label{sec:alpha}
The goal of this section is to prove all six enumerative results collected in Table~\ref{tab:3-alpha all}. We divide them into two subsections.

\subsection{The three non-Fibonacci cases}
Let us first deal with the cases of $\pi=123$, $132$, and $312$.

\begin{theorem}\label{thm:123-alpha}
For $n\ge 2$, we have $a_n(123,\alpha)=2^n-n$, thus Conjecture~\ref{conj:Archer-Laudone} (1) holds true.
\end{theorem}

\begin{proof}
For every $n\ge 2$, we write $a_n:=a_n(123,\alpha)$ and aim to derive a recurrence for $a_n$. Take any permutation $\sigma\in\fS_n(123,\alpha)$. If $\sigma_n=n$ then to avoid $123$ we must have $\sigma_1\cdots\sigma_{n-1}=(n-1)(n-2)\cdots 2\, 1$, which contributes $1$ to the count of $a_n$.

Otherwise, we are in case (2) of Lemma~\ref{lem:char of alpha} and thus can assume the decomposition
\begin{align}
	\sigma=\tau_1\cdots\tau_i \eta_1'\cdots\eta_{n-k}'\tau_{i+2}\cdots\tau_{k}k,\label{decomp}
\end{align}
for some $0\le i < k\le n-1$, where
\begin{align*}
	&\tau:=\tau_1\cdots\tau_i\,k\,\tau_{i+2}\cdots\tau_k\in \fS_k(\alpha), \text{ and}\\
	&\eta:=\eta_1\cdots\eta_{n-k}\in \fS_{n-k}(\alpha) \text{ with } \eta'_j=\eta_j+k,\; 1\le j\le n-k.
\end{align*}
 As we take into account the further restriction that $\sigma$ avoids $123$, there are two subcases to consider.

\begin{itemize}
\item Case 1: $i=0$. We see \eqref{decomp} becomes $\sigma=\eta_1'\cdots\eta_{n-k}'\tau_{2}\cdots\tau_{k}k$. The presence of $k$ combined with the $123$-avoiding condition results in $\tau_{2}\cdots\tau_{k}=(k-1)(k-2)\cdots 1$. In particular $\tau_{2}\cdots\tau_{k}k$ is already $123$-avoiding. Moreover, since every integer contained in the prefix $\eta_1'\cdots\eta_{n-k}'$ is larger than every integer from the suffix $\tau_2\cdots\tau_k k$, any occurrence of pattern $123$ must be completely confined in $\eta_1'\cdots\eta_{n-k}'$. Consequently, it suffices to require that $\eta$ belongs to $\fS_{n-k}(123,\alpha)$. Conversely, taking any $\eta\in\fS_{n-k}(123,\alpha)$, increasing each of its entries by $k$, and appending $(k-1)(k-2)\cdots 1\, k$ yields a unique permutation in Case 1. Thus, for a fixed $k$, the total contribution from this case is $a_{n-k}$.

\item Case 2: $i>0$. Again, applying the $123$-avoiding condition, we can deduce that $\eta=(n-k)(n-k-1)\cdots 1$ and $\tau_1\cdots\tau_i\tau_{i+2}\cdots\tau_k=(k-1)(k-2)\cdots 1$, with the value of $i$ ($1\le i\le k-1$) dictating the cut-off between $\tau_i$ and $\tau_{i+2}$. Hence for a fixed $k\ge 2$, the contribution from this case is $k-1$.
\end{itemize}

Summarizing all cases, we conclude that
\[
a_n
=
1+\sum_{k=1}^{n-1}a_{n-k}+\sum_{k=2}^{n-1}(k-1)
=
1+\sum_{m=1}^{n-1}a_m+\frac{(n-1)(n-2)}{2},
\]
which can be iterated to produce the desired recurrence relation that holds for all $n\ge 2$:
\begin{align}\label{rec for a_n}
a_n=2a_{n-1}+n-2.
\end{align}
It is clear that the sequence $\{2^n-n\}_{n\ge 2}$ also satisfies \eqref{rec for a_n}. Together with the base case $a_2=2=2^2-2$, this completes the proof by induction.
\end{proof}

\begin{theorem}\label{thm:132-alpha}
For $n\ge 2$, we have $a_n(132,\alpha)=a_n(312,\alpha)=C_n$, the $n$-th Catalan number.
\end{theorem}
\begin{proof}
Let us denote $b_n:=a_n(132,\alpha)$ and $c_n:=a_n(312,\alpha)$. Clearly $b_0=c_0=1$ and $b_1=c_1=1$. It suffices to show that for $n\ge 2$,
\begin{align}
	b_n &= b_{n-1}+\sum_{k=1}^{n-1} b_{k-1}b_{n-k},\label{rec:b_n}\\
	c_n &= c_{n-1}+\sum_{k=1}^{n-1} c_{k-1}c_{n-k},\label{rec:c_n}
\end{align}
each of which is the well-known convolutive recursion for Catalan numbers. Indeed, take any permutation $\sigma\in\fS_n(132,\alpha)$, we apply Lemma~\ref{lem:char of alpha} together with the $132$-avoiding constraint to deduce the following two cases. 
\begin{itemize}
	\item $\sigma_n=n$ and $\sigma_1\sigma_2\cdots\sigma_{n-1}\in\fS_{n-1}(132,\alpha)$. Conversely, appending $n$ to each permutation in $\fS_{n-1}(132,\alpha)$ gives rise to an $n$-permutation that avoids both $132$ and $\alpha$. This case explains the term $b_{n-1}$ in \eqref{rec:b_n}.
	\item $\sigma=\eta_1'\cdots\eta_{n-k}'\tau_{1}\cdots\tau_{k-1}k$ for some $1\le k\le n-1$, such that 
	$$\tau_1\cdots\tau_{k-1}\in\fS_{k-1}(132,\alpha) \text{ and } \eta_1\cdots\eta_{n-k}\in\fS_{n-k}(132,\alpha),$$ 
	with $\eta_j':=\eta_j+k$ for every $1\le j\le n-k$. Conversely, given any two permutations $\tau\in\fS_{k-1}(132,\alpha)$ and $\eta\in\fS_{n-k}(132,\alpha)$, the skew sum $\eta\ominus(\tau\, k)$ is in $\fS_n(132,\alpha)$. Hence this case corresponds to the summation in \eqref{rec:b_n}.
\end{itemize}
Combining the above two cases we arrive at \eqref{rec:b_n}. The proof of \eqref{rec:c_n} is analogous and thus omitted.
\end{proof}

\begin{remark}
It is worth noting that a bijection, say 
$$\phi:\fS_n(132,\alpha)\to \fS_n(312,\alpha),$$
can be recursively constructed to show that $a_n(132,\alpha)=a_n(312,\alpha)$. Initially, we set $\phi(1)=1$, $\phi(12)=12$, and $\phi(21)=21$. Suppose $\phi$ is already defined for those permutations of length smaller than a certain $n\ge 3$, then for a given $\sigma=\eta\ominus(\tau k)\in\fS_n(132,\alpha)$ for some $1\le k\le n$, $\eta\in\fS_{n-k}(132,\alpha)$, and $\tau\in\fS_{k-1}(132,\alpha)$, we define 
$$\phi(\sigma)=(\phi(\tau)\oplus\phi(\eta))k.$$

Furthermore, it is well-known that $|\fS_n(132)|=C_n$. In view of the trivial inclusion $\fS_n(132,\alpha)\subseteq\fS_n(132)$, we see that actually the following relation holds for all $n\ge 1$:
\begin{align}\label{132-Catalan}
	\fS_n(132)=\fS_n(132,\alpha).
\end{align}
By directly establishing \eqref{132-Catalan} from the definition of the arrow pattern $\alpha$ and applying the bijection $\phi$, we arrive at an alternative proof of Theorem~\ref{thm:132-alpha}.
\end{remark}

\subsection{The odd-indexed Fibonacci numbers}

In this subsection we show that the remaining three classes (i.e., for $\pi=213$, $231$, or $321$) are all enumerated by the odd-indexed Fibonacci numbers $\{F_{2n-1}\}_{n\ge 1}=\{1,2,5,13,34,89,\ldots\}$, thereby completing the enumerations summarized in Table~\ref{tab:3-alpha all}. Let $F(x):=1+\sum_{n\ge 1}F_{2n-1}x^n$ be the generating function of the odd-indexed Fibonacci numbers. It is known (see for instance \cite[Eq.~(2.2.7)]{Wil2005}) that
\begin{align}\label{gf:oddFibo}
	F(x) &= \frac{1-2x}{1-3x+x^2}.
\end{align}

For the two cases with $\pi=213$ and $\pi=231$, we are able to establish a much stronger result; see Theorem~\ref{thm:213-231-q} below. To that end, we need the following version of $(q,t)$-Catalan numbers introduced by Fu, Tang, Han, and Zeng~\cite{FTHZ2019}. For every $n\ge 1$, let $C_n(t,q)$ be the coefficient of $x^n$ in the following continued fraction expansion
\begin{align}\label{def-q-nara}
C(t,q,x):=\sum_{n=0}^\infty C_n(t,q) x^n=
\cfrac{1}{
1-\cfrac{x}{1-\cfrac{tx}{
\cfrac{\ddots}{1-\cfrac{q^{k-1}x}{1-\cfrac{tq^{k-1}x}{\ddots}
}}}}}\, .
\end{align}
It was shown in \cite[Theorem~1.1]{FTHZ2019} that $C_n(t,q)$ has ten interpretations in terms of the distributions of various statistics over pattern avoiding permutations. In particular, the following interpretations (corresponding to pairs \#1 and \#9 in \cite[Table~1]{FTHZ2019}) play a key role in our proof of Theorem~\ref{thm:213-231-q}.
\begin{theorem}
For every $n\ge 1$, we have
\begin{align}\label{id:213-231-q}
C_n(t,q) = \sum_{\sigma\in\fS_n(213)}t^{\des(\sigma)}q^{\acb(\sigma)} = \sum_{\sigma\in\fS_n(231)}t^{\des(\sigma)}q^{\acb(\sigma)},
\end{align}
where $\des(\sigma):=|\{i\in[n-1]:\sigma_i>\sigma_{i+1}\}|$ is the number of \emph{descents} of $\sigma$, and
\begin{align*}
	\acb(\sigma) := |\{(i,j)\in[n]^2: 1<i+1<j\le n,~\sigma_i<\sigma_j<\sigma_{i+1}\}|.
\end{align*}
\end{theorem}

Note that $\acb$ is usually referred to as a \emph{vincular pattern}; see \cite[Chapter 7.1]{Kit11} for further information. We let $\alpha(\sigma)$ denote the number of occurrences of arrow pattern $\alpha$ in $\sigma$, and introduce three generating functions:
\begin{align*}
	R^{213}(q,x) &:= 1+\sum_{n\ge 1}x^n\sum_{\sigma\in\fS_n(213)}q^{\alpha(\sigma)},\\
	R^{231}(q,x) &:= 1+\sum_{n\ge 1}x^n\sum_{\sigma\in\fS_n(231)}q^{\alpha(\sigma)},\\
	R^{321}(q,x) &:= 1+\sum_{n\ge 1}x^n\sum_{\sigma\in\fS_n(321)}q^{\alpha(\sigma)}.
\end{align*}
As evidenced by Table~\ref{tab:3-alpha all}, the enumeration of the three cases corresponding to $\pi=213$, $231$, and $321$ consistently results in the odd-indexed Fibonacci numbers. This equinumerosity could be succinctly rephrased as
\begin{align*}
R^{213}(0,x)=R^{231}(0,x)=R^{321}(0,x),
\end{align*}
wherein the first equality can be strengthened as follows.

\begin{theorem}\label{thm:213-231-q}
We have
\begin{align}\label{id:gf-213-231-q}
R^{213}(q,x)=R^{231}(q,x)=\frac{1}{1-xC(1,q,x)}.
\end{align}
In particular, for every $n\ge 1$, we have
\begin{align}\label{id:213-231-Fibo}
a_n(213,\alpha)=a_n(231,\alpha)=F_{2n-1}.
\end{align}
\end{theorem}

\begin{proof}
We begin with a decomposition for any given permutation $\sigma\in\fS_n$ with $\LRMax(\sigma)=\{m_1,\ldots,m_k\}$, such that $\sigma$ is either $213$-avoiding or $231$-avoiding. Each of these two cases confirms half of \eqref{id:gf-213-231-q}.
\begin{itemize}
	\item If $\sigma\in\fS_n(213)$, then it can be uniquely decomposed as
	$$\sigma=m_1m_2\cdots m_ku_k\cdots u_2u_1,$$
	such that for $1\le i\le k$, the union of the entries contained in the subword $u_i$ is precisely the interval $(m_{i-1},m_i)$ (with $m_0:=0$ as a convention), and their standardizations must all be $213$-avoiding as well; see the left diagram in Fig.~\ref{fig:decomp} for an illustration of the case $k=3$. This decomposition also makes the following relation evident:
	\begin{align}\label{id:stat-decomp}
	\alpha(\sigma)=\sum_{i=1}^{k}\acb(u_i).
	\end{align}
	Turning this identity into functional equation, we deduce that
	\begin{align*}
	R^{213}(q,x) &= 1+x\cdot \sum_{\tau \text{ avoids }213}q^{\acb(\tau)}x^{|\tau|}+\left(x\cdot \sum_{\tau \text{ avoids }213}q^{\acb(\tau)}x^{|\tau|}\right)^2+\cdots\\
	&=\frac{1}{1-x\sum_{n\ge 0}x^n\sum_{\tau\in\fS_n(213)}q^{\acb(\tau)}}\\
	&=\frac{1}{1-xC(1,q,x)},
	\end{align*}
	where we have applied \eqref{id:213-231-q} for the last step.
	\item If $\sigma\in\fS_n(231)$, then it can be uniquely decomposed as
	$$\sigma=m_1u_1m_2u_2\cdots m_ku_k,$$
	such that for $1\le i\le k$, the union of the entries contained in the subword $u_i$ is precisely the interval $(m_{i-1},m_i)$ (with $m_0:=0$ as a convention), and their standardizations must all be $231$-avoiding as well; see the right diagram in Fig.~\ref{fig:decomp} for an illustration of the case $k=3$. The same relation \eqref{id:stat-decomp} still holds in this case, which is combined with the $t=1$ case of \eqref{id:213-231-q} to prove that $R^{231}(q,x)=\frac{1}{1-xC(1,q,x)}$.
\end{itemize}

Next, to prove \eqref{id:213-231-Fibo}, it suffices to show that these three sequences share the common generating function \eqref{gf:oddFibo}. On the one hand, setting $q=0$ in \eqref{id:gf-213-231-q} produces
\begin{align}\label{id:213-231-q=0}
\sum_{n\ge 0}a_n(213,\alpha)x^n=\sum_{n\ge 0}a_n(231,\alpha)x^n=\frac{1}{1-xC(1,0,x)}.
\end{align} 
On the other hand, setting $t=1$ and $q=0$ in \eqref{def-q-nara} yields the terminated continued fraction
\begin{align*}
C(1,0,x)=\sum_{n=0}^{\infty} C_n(1,0)x^n=\cfrac{1}{1-\cfrac{x}{1-x}}=\frac{1-x}{1-2x}.
\end{align*}
Plugging this back to \eqref{id:213-231-q=0}, we deduce that
$$\sum_{n\ge 0}a_n(213,\alpha)x^n=\sum_{n\ge 0}a_n(231,\alpha)x^n=\dfrac{1}{1-x\dfrac{1-x}{1-2x}}=\frac{1-2x}{1-3x+x^2},$$
which indeed agrees with \eqref{gf:oddFibo}.
\end{proof}

\begin{figure}
\begin{center}
\begin{tabular}{ccc}
\begin{tikzpicture}[scale=1.1]
  \tikzset{
    grid/.style={ draw, step=1cm, black!100, thin },
    graycell/.style={ fill=gray!50, draw=none, minimum width=1cm, minimum height=1cm, anchor=south west }
  }

  \fill[graycell] (0,1) rectangle (1,3);
  \fill[graycell] (1,2) rectangle (2,3);
  \fill[graycell] (0,0) rectangle (4,1);
  \fill[graycell] (1,1) rectangle (3,2);
  \fill[graycell] (3,2) rectangle (5,3);
  \fill[graycell] (4,1) rectangle (5,2);
  
  \draw[grid] (0,0) grid (5,3);

  \draw[thick] (3.1, 1.1) rectangle (3.9, 1.9);
  \draw[thick] (2.1, 2.1) rectangle (2.9, 2.9);
  \draw[thick] (4.1, 0.1) rectangle (4.9, 0.9);

  \node[anchor=center] at (4.5, .5) {$u_1$};
  \node[anchor=center] at (3.5, 1.5) {$u_2$};
  \node[anchor=center] at (2.5, 2.5) {$u_3$};

  \filldraw[black] (0,1) circle (2.5pt);
  \node[anchor=west] at (-.6,1.2) {\footnotesize{$m_1$}};
  \filldraw[black] (1,2) circle (2.5pt);
  \node[anchor=west] at (.4,2.25) {\footnotesize{$m_2$}};
  \filldraw[black] (2,3) circle (2.5pt);
  \node[anchor=west] at (1.7,3.25) {\footnotesize{$m_3$}};
  \node at (2.5,-.5) {$\sigma=m_1m_2m_3u_3u_2u_1\in\fS_n(213)$};
\end{tikzpicture}

&

\ \ \ 

&

\begin{tikzpicture}[scale=1.1]
  \tikzset{
    grid/.style={ draw, step=1cm, black!100, thin },
    graycell/.style={ fill=gray!50, draw=none, minimum width=1cm, minimum height=1cm, anchor=south west }
  }

  \fill[graycell] (0,1) rectangle (1,3);
  \fill[graycell] (1,2) rectangle (2,3);
  \fill[graycell] (1,0) rectangle (3,1);
  \fill[graycell] (2,1) rectangle (3,2);
  
  \draw[grid] (0,0) grid (3,3);

  \draw[thick] (1.1, 1.1) rectangle (1.9, 1.9);
  \draw[thick] (2.1, 2.1) rectangle (2.9, 2.9);
  \draw[thick] (0.1, 0.1) rectangle (0.9, 0.9);

  \node[anchor=center] at (0.5, 0.5) {$u_1$};
  \node[anchor=center] at (1.5, 1.5) {$u_2$};
  \node[anchor=center] at (2.5, 2.5) {$u_3$};

  \filldraw[black] (0,1) circle (2.5pt);
  \node[anchor=west] at (-.6,1.2) {\footnotesize{$m_1$}};
  \filldraw[black] (1,2) circle (2.5pt);
  \node[anchor=west] at (.4,2.25) {\footnotesize{$m_2$}};
  \filldraw[black] (2,3) circle (2.5pt);
  \node[anchor=west] at (1.7,3.25) {\footnotesize{$m_3$}};
  \node at (1.5,-.5) {$\sigma=m_1u_1m_2u_2m_3u_3\in\fS_n(231)$};
\end{tikzpicture}
\end{tabular}
\caption{Two decompositions for a permutation $\sigma$ with $\LRMax(\sigma)=\{m_1,m_2,m_3\}$}\label{fig:decomp}
\end{center}
\end{figure}
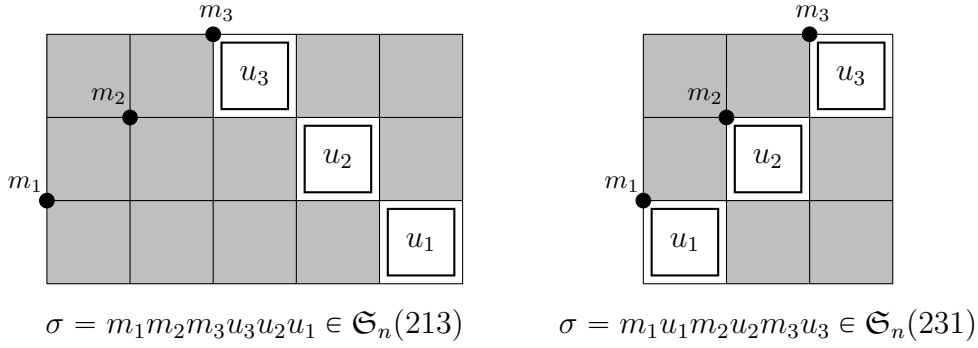

The rest of this subsection is devoted to the final case with $\pi=321$.

\begin{theorem}\label{thm:321-alpha}
For $n\ge 2$, we have $a_n(321,\alpha)=F_{2n-1}$, thus Conjecture~\ref{conj:Archer-Laudone} (2) holds true.
\end{theorem}
\begin{proof}
Let us denote $h_n:=a_n(321,\alpha)$ with initial values $h_0=h_1=1$, and define $H(x):=\sum_{n\ge 0}h_nx^n$. 
We consider two cases relying on the characterization of $\alpha$-avoiding permutations given by Lemma~\ref{lem:char of alpha}, and then derive a recurrence relation that is satisfied by $h_n$.

Take any permutation $\sigma\in\fS_n(321,\alpha)$. The first case with $\sigma_n=n$ is clear since we have $\sigma_1\cdots\sigma_{n-1}\in\fS_{n-1}(321,\alpha)$, and the contribution from this case is $h_{n-1}$. The second case with $\sigma_n=k<n$ is more intricate to analyze. We begin with the familiar decomposition for some $0\le i<k\le n-1$:
\begin{align*}
\sigma &= \tau_1\cdots\tau_i \eta_1\cdots\eta_{n-k}\tau_{i+2}\cdots\tau_k\, k,
\end{align*}

where $\eta_1\cdots\eta_{n-k}=(k+1)\cdots n$ to avoid $321$. The $321$-avoidance also forces $\tau^{(2)}:=\tau_{i+2}\cdots\tau_k$ to be monotonically increasing. Further analysis is needed for $\tau^{(1)}:=\tau_1\cdots\tau_i$. Recall that Lemma~\ref{lem:char of alpha}~(2) requires that $\tau^{(1)}k\tau^{(2)}\in\fS_k(\alpha)$. There are two subcases.
\begin{enumerate}
	\item $\tau^{(2)}=\varnothing$. In this case $\tau^{(1)}$ can be any permutation in $\fS_{k-1}(321,\alpha)$, and conversely, any permutation taken from $\fS_{k-1}(321,\alpha)$ corresponds to a unique permutation in $\fS_n(321,\alpha)$ when appended by the sequence $(k+1)\cdots nk$. For a fixed $k$, this yields a contribution of $h_{k-1}$.
	\item $i\le k-2$ and $\tau^{(2)}\neq\varnothing$. Let us fix the value of $\tau_{i+2}=:j$. Recall that $j=\tau_{i+2}<\tau_{i+3}<\cdots<\tau_k<k$. We claim that $\tau^{(1)}$, consisting of integers in $[k-1]\setminus\{\tau_{i+2},\ldots,\tau_k\}$, satisfies the following two conditions.
	\begin{description}
		\item[(i)] Every integer $x\in [j,k-1]\setminus\{\tau_{i+2},\ldots,\tau_k\}$ is a left-to-right maximum in $\tau^{(1)}$.
		\item[(ii)] The prefix $\tau_1\cdots\tau_{j-1}$ of $\tau^{(1)}$ is a permutation of $[j-1]$ that avoids both $321$ and $\alpha$.
	\end{description}
	Conversely, given any permutation $\tau^{(3)}=\tau_1\cdots\tau_{j-1}\in\fS_{j-1}(321,\alpha)$, we can recover a permutation in case (2) with $\tau_{i+2}=j$ by placing each integer $x\in [j+1,k-1]$ into either $\tau^{(1)}$ (between $\tau^{(3)}$ and $\eta_1\cdots\eta_{n-k}$) or $\tau^{(2)}$. To do this, they must be placed such that they form monotonically increasing subsequences in both $\tau^{(1)}$ and $\tau^{(2)}$. Consequently, the total contribution from this case (2) for fixed $k$ and $j$ is $2^{k-1-j}\cdot h_{j-1}$. In summary of all the cases, we get the following recurrence relation for $n\ge 2$:
	\begin{align*}
	h_n &= h_{n-1}+\sum_{k=1}^{n-1}h_{k-1}+\sum_{k=2}^{n-1}\sum_{j=1}^{k-1}2^{k-1-j}h_{j-1} \\
	&= \sum_{i=0}^{n-1}h_i+\sum_{k=0}^{n-3}\sum_{j=0}^{k}2^{k-j}h_j,
	\end{align*}
	It is then routine to deduce the functional equation for $H(x)$:
	\begin{align*}
	H(x)-x-1 &= x\left(\frac{H(x)}{1-x}-1\right)+x^3\frac{1}{1-x}\cdot \frac{H(x)}{1-2x}.
	\end{align*}
	Solving for $H(x)$ we conclude that $H(x)=\frac{1-2x}{1-3x+x^2}=F(x)$, as desired. 

	To finish the proof, it remains to prove claims~(i) and (ii). To see (i), we suppose on the contrary that there exists an integer $x\in [j,k-1]\setminus\{\tau_{i+2},\ldots,\tau_k\}$ that is not a left-to-right maximum, hence there is another integer $y>x$ that is to the left of $x$ in $\tau^{(1)}$, but then the triple $(y,x,j)$ witnesses a $321$ pattern in $\sigma$, leading to a contradiction. 

	Next, suppose $[j,k-1]\setminus\{\tau_{i+2},\ldots,\tau_k\}=\{x_1,\ldots,x_{i-j+1}\}_<$, then claim (ii) is equivalent to saying that $x_1x_2\cdots x_{i-j+1}(k+1)$ is a factor\footnote{A factor of a word is a contiguous block of letters.} of $\sigma$. Suppose on the contrary that there is a nonempty factor $\nu$ sitting between $x_{\ell}$ and $x_{\ell+1}$, for some $1\le \ell\le i-j+1$ (set $x_{i-j+2}=k+1$ as a convention). Note that $\nu$ is composed of integers from $[j-1]$, so in particular they are all smaller than $j$. Denoting the rightmost letter of $\nu$ as $y$, we note that in the cycle notation of the preimage $\Phi^{-1}(\sigma)$, the cycle that begins with $x_{\ell}$ must ends with $y$, rendering the triple $(x_{\ell},y,j)$ an occurrence of the arrow pattern $\alpha$ in $\sigma$. This is a contradiction so we have both claims verified and the proof is now complete.
\end{enumerate}
\end{proof}
	
Our proof of Theorem~\ref{thm:321-alpha} implies a refinement by a statistic that we introduce next.
	
\begin{definition}\label{def:q-Fibo}
	For a permutation $\sigma\in\fS_n$, we denote by $\tail(\sigma)$ the smallest integer $j$, $0\le j\le n-1$, such that $\sigma_{n-j}\neq n-j$. When no such $j$ exists, i.e., $\sigma=12\cdots n$ is the identity permutation, we set $\tail(\sigma)=n$. Said in another way, $\tail(\sigma)$ counts the number of ending fixed points of $\sigma$.
\end{definition}
	
	Let us introduce the generating function 
	$$f_n(q):=\sum_{\sigma\in \fS_n(321,\alpha)}q^{\tail(\sigma)}.$$

	The first few of these polynomials are given by $f_0(q):=1$, $f_1(q)=q$, $f_2(q)=q^2+1$, $f_3(q)=q^3+q+3$, etc. We have the following $q$-extension of Conjecture~\ref{conj:Archer-Laudone}~(2).
	\begin{corollary}
		For $n\ge 1$, we have:
		\begin{align}\label{eq:q-odd-Fibo}
			f_n(q) &= q^{n}+\sum_{k=1}^{n-1}F_{2k}q^{n-k-1}.
		\end{align}
	\end{corollary}
	\begin{proof}
	For a permutation $\sigma\in\fS_n(321,\alpha)$, the two cases of the ending element $\sigma_n$ clearly yields the following recurrence
	\begin{align}\label{rec:fq-1}
		f_n(q) &= qf_{n-1}(q)+f_n(0),
	\end{align}
	which can be iterated to give us
	\begin{align}
		f_n(q) &= \sum_{i=0}^nq^{i}f_{n-i}(0)=f_0(0)q^n+f_1(0)q^{n-1}+\sum_{i=0}^{n-2}q^{i}f_{n-i}(0)\nonumber \\
		&= q^n+\sum_{i=1}^{n-1}q^{n-1-i}f_{i+1}(0).\label{rec:fq-2}
	\end{align}
	On the other hand, plugging $q=1$ in \eqref{rec:fq-1} and applying Theorem~\ref{thm:321-alpha}, we deduce that 
	$$f_n(0)=f_n(1)-f_{n-1}(1)=F_{2n-1}-F_{2n-3}=F_{2n-2}.$$
	Plugging this back to \eqref{rec:fq-2}, we arrive at \eqref{eq:q-odd-Fibo}.
	\end{proof}
	\begin{remark}
	Setting $q=1$ in \eqref{eq:q-odd-Fibo} and applying Theorem~\ref{thm:321-alpha} for the left-hand side, we obtain a known identity for the Fibonacci numbers. For $n\ge 2$, we have that
		\begin{align*}
			F_2+F_4+\cdots+F_{2n-2} &= F_{2n-1}-1.
		\end{align*}
	See~\cite[Id.~12]{BQ2003} for an interesting combinatorial approach via tilings.
	\end{remark}

\section{Two proofs of Conjecture~\ref{conj:Archer-Laudone}~(3)}\label{sec:beta}
We provide two proofs for the remaining item (3) of Conjecture~\ref{conj:Archer-Laudone}, which involves a different arrow pattern $\beta:=(12;1\to 2)$. 

\subsection{A direct bijective proof}
We first characterize in a more explicit way the permutations that avoid both $321$ and $\beta$, then we build a bijection $\psi$ from $\cM_n$ to $\fS_n(321,\beta)$. Here $\cM_n$ denotes the set of Motzkin paths of length $n$, whose definition we recall next.
\begin{definition}\label{def:M-path}
	For $n\ge 0$, a \emph{Motzkin path} of length $n$ is a lattice path from $(0,0)$ to $(n,0)$ that never falls below the $x$-axis, consisting of \emph{up ($\nearrow$)} steps $\rU=(1,1)$, \emph{down ($\searrow$)} steps $\rD=(1,-1)$, and \emph{level ($\rightarrow$)} steps $\rL=(1,0)$. If a Motzkin path contains no level steps, we also call it a \emph{Dyck path} of semilength $n/2$.
\end{definition}

In what follows, Motzkin paths are usually expressed as a word composed of letters $\rU$, $\rD$, and $\rL$. For example, the word $w=\rU\rU\rD\rL\rD\rL\rU\rD$ represents the following Motzkin path of length $8$:
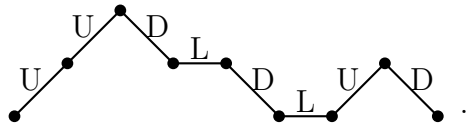
\begin{align*}
		\vcenter{\hbox{\begin{tikzpicture}[line width=0.8pt,scale=0.7]
					\coordinate (O) at (0,0);
					\draw[thick] (O)--++(1,1)--++(1,1)--++(1,-1)--++(1,0)--++(1,-1)--++(1,0)--++(1,1)--++(1,-1);
					\filldraw (O) circle(0.5ex) ++(1,1)circle(0.5ex) ++(1,1)circle(0.5ex) ++(1,-1)circle(0.5ex)
					++(1,0)circle(0.5ex) ++(1,-1)circle(0.5ex) ++(1,0)circle(0.5ex) ++(1,1)circle(0.5ex) ++(1,-1)circle(0.5ex);
					\path (0.3,0.7) node {$\rU$} ++(1,1) node {$\rU$} ++(1.4,0) node {$\rD$} ++(0.8,-0.4) node {$\rL$} ++(1.2,-0.6) node {$\rD$} ++(0.8,-0.4) node {$\rL$} ++(0.8,0.4) node {$\rU$} ++(1.4,0) node {$\rD$} ++(0.8,-0.6) node {.};
		\end{tikzpicture}}}
\end{align*}

Recall that a permutation $\sigma$ is called an \emph{involution} if and only if $\sigma^{-1}=\sigma$. Equivalently, $\sigma$ is an involution if and only if there are only $1$-cycles (fixed points) and $2$-cycles (transpositions) in its cycle notation. We denote by $\rI_n$ the set of involutions of length $n$. An involution is said to be \emph{non-nesting} if it avoids the classical pattern $4321$. The introduction of non-nesting involutions is motivated by the following alternative characterization of $\fS_n(321,\beta)$.

\begin{proposition}\label{prop:char 321-beta}
A permutation $\sigma$ avoids simultaneously the classical pattern $321$ and the arrow pattern $\beta$, if and only if the preimage $\hat{\sigma}:=\Phi^{-1}(\sigma)$ is a non-nesting involution.
\end{proposition}
\begin{proof}
Suppose $\hat{\sigma}$ is written in its standard cycle notation. By the definition of arrow pattern, $\sigma$ avoids the arrow pattern $\beta=(12;1\to 2)$, if and only if inside each cycle of $\hat{\sigma}$ the elements are monotonically decreasing. Assume that is the case, then $\sigma$ further avoids $321$ if and only if 
\begin{enumerate}[1)]
\item each cycle of $\hat{\sigma}$ has at most two elements;
\item denoting the $2$-cycles of $\hat{\sigma}$ as $(d_1\,u_1),(d_2\,u_2),\ldots,(d_m\,u_m)$, with $d_1<d_2<\cdots<d_m$ and $d_i>u_i$ for all $1\le i\le m$, then we have $u_1<u_2<\cdots<u_m$.
\end{enumerate}
Note that condition 1) amounts to saying that $\hat{\sigma}$ is an involution, while condition 2) is equivalent to requiring that the involution $\hat{\sigma}$ is non-nesting.
\end{proof}

It is known that $\rI_n(4321)$, i.e., the set of $4321$-avoiding involutions is enumerated by the $n$-th Motzkin number $M_n$; see for instance \cite{Gui1995,Jag2003}. We supply here a proof of this result relying on a bijection that is essentially Biane's bijection over $\fS_n$ given in \cite{Bia1993} restricted to $\rI_n(4321)$. Our description of this bijection is adapted from \cite[Theorem~3]{BBS2011}.

\begin{theorem}\label{thm:Motzkin and nonnesting involution}
There exists a bijection $\Theta:\cM_n\to \rI_n(4321)$ such that
\begin{align}\label{stat-Theta}
	\nou(w) &= \cc(\Theta(w)),
\end{align}
for any Motzkin path $w\in\cM_n$. Here $\nou(w)$ is the number of $\rU$'s contained in $w$, and $\cc(\sigma)$ denotes the number of $2$-cycles in the cycle notation of a permutation $\sigma$.
\end{theorem}
\begin{proof}
For a given Motzkin path $w=w_1\cdots w_n\in\cM_n$, where $w_i\in\{\rU,\rD,\rL\}$, we scan it from left to right, getting three subsets that record the positions of $\rU$'s, $\rD$'s, and $\rL$'s, respectively:
\begin{align*}
	\cU(w) &= \{u_1,u_2,\ldots,u_m\},\\
	\cD(w) &= \{d_1,d_2,\ldots,d_m\},\\
	\cL(w) &= \{l_1,l_2,\ldots,l_{n-2m}\}.
\end{align*}
Next, we pair $\cU(w)$ with $\cD(w)$ to get $2$-cycles, namely $(d_1\,u_1),(d_2\,u_2),\ldots,(d_m\,u_m)$, while each element in $\cL(w)$ forms a $1$-cycle by itself. Notice that $w$ being a Motzkin path (never going below $x$-axis) ensures that $d_i>u_i$ for every $1\le i\le m$. Therefore, lining up all these $2$-cycles and $1$-cycles increasingly (with respect to the leading element of the cycle) from left to right, we arrive at the standard cycle notation of a certain permutation, which we set as the image $\Theta(w)$. It is easy to see from our construction that $\Theta(w)$ satisfies conditions 1) and 2) in Proposition~\ref{prop:char 321-beta}, i.e., it is a non-nesting involution thus $\Theta$ is well-defined. 

The relation \eqref{stat-Theta} between statistics $\nou$ and $\cc$ is evident. Moreover, it is clear how to construct the inverse mapping $\Theta^{-1}:\rI(4321)\to \cM_n$. Namely, given a non-nesting involution $\sigma$ written in its standard cycle notation, we collect repectively the smaller element from each of the $2$-cycles as a subset $\cU(\sigma)$, the larger element from each $2$-cycle as another subset $\cD(\sigma)$, and the $1$-cycles as a third subset $\cL(\sigma)$. Then, viewing $\cU(\sigma)$, $\cD(\sigma)$, $\cL(\sigma)$ as the set of positions for up steps, down steps, and level steps, respectively, we get the preimage path $\Theta^{-1}(\sigma)$.
\end{proof}
Take the previous path $w=\rU\rU\rD\rL\rD\rL\rU\rD$ for an example, we see that $\cU(w)=\{1,2,7\}$, $\cD(w)=\{3,5,8\}$, and $\cL(w)=\{4,6\}$. Hence we get $\Theta(w)=(3\,1)(4)(5\,2)(6)(8\,7)$, which is indeed a non-nesting involution of length $8$.

Now we are in a position to give the first proof of Conjecture~\ref{conj:Archer-Laudone}~(3).
\begin{theorem}\label{thm:beta-Biane}
There exists a bijection $\psi: \cM_n \to \fS_n(321,\beta)$, such that
\begin{align}\label{stat-psi}
	\nou(w) &= \des(\psi(w))
\end{align}
for every $w \in\cM_n$. In particular, Conjecture~\ref{conj:Archer-Laudone}~(3) holds true.
\end{theorem}	

\begin{proof}
It suffices to define $\psi$ as the composition map $\psi:=\Phi\circ\Theta:\cM_n\to \fS_n(321,\beta)$; see the commutative diagram in Fig.~\ref{triangle}. In view of Proposition~\ref{prop:char 321-beta} and Theorem~\ref{thm:Motzkin and nonnesting involution}, the mapping $\psi$ is the composition of two bijections thus is itself a bijection from $\cM_n$ to $\fS_n(321,\beta)$ for every $n\ge 1$. Furthermore, we see that $\sigma_i$ and $\sigma_{i+1}$ forms a descent pair in $\sigma$, if and only if they form a $2$-cycle in $\Phi^{-1}(\sigma)$. Combining this with \eqref{stat-Theta} we can deduce \eqref{stat-psi}, thereby completing the proof.
\end{proof}

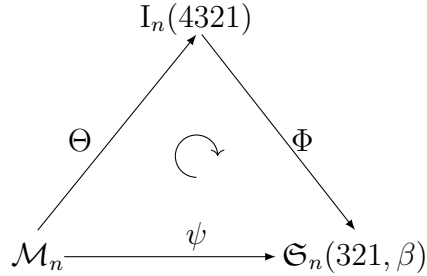
\begin{figure}[h!]
\begin{tikzpicture}[scale=0.7]
\draw(0,0) node{$\cM_{n}$};
\draw(6,0) node{$\fS_{n}(321,\beta)$};
\draw(3,4.5) node{$\rI_n(4321)$};

\draw(.8,2.2) node{$\Theta$};
\draw(5,2.2) node{$\Phi$};
\draw(3,0.4) node{$\psi$};

\draw[->] (3,1.5) arc (270 : 0 : 0.4cm);

\draw[-latex] (0,0.5)->(3,4.2);
\draw[-latex] (0.5,0)->(4.5,0);
\draw[-latex] (3.1,4.2)->(6,0.5);
\end{tikzpicture}
\caption{The composition $\psi=\Phi\circ\Theta$ \label{triangle}}
\end{figure}

Since one can insert $n-2k$ level steps into a Dyck path of semilength $k$ (counted by $C_k$) in $\binom{n}{2k}$ distinct ways, to obtain an $n$-Motzkin path with exactly $k$ up steps, we immediately get the following generating function for the Eulerian (i.e., for the Eulerian statistic $\des$) distribution over $\fS_n(321,\beta)$. This is a polynomial refinement of Conjecture~\ref{conj:Archer-Laudone}~(3).
\begin{corollary}
For every $n\ge1$,
\begin{align}\label{eq: des-321-beta=Motzkin poly}
	& \sum_{\sigma\in \fS_n(321,\beta)} t^{\des(\sigma)} = \sum_{w\in\cM_n} t^{\nou(w)} =
	\sum_{k=0}^{\lfloor n/2\rfloor}\binom{n}{2k}C_k t^k=\frac{1}{n+1}\sum_{k=0}^{\lfloor n/2\rfloor}\binom{n+1}{k,k+1,n-2k}t^k.
\end{align}
\end{corollary}

\subsection{A second proof and the ``UUU'' statistic over Dyck paths}
Our second proof of Conjecture~\ref{conj:Archer-Laudone}~(3) approaches the problem from a different perspective, yet remains fundamentally bijective in nature. Seeing the trivial inclusion $\fS_n(321,\beta)\subseteq\fS_n(321)$, we treat $\beta$ as a frequency counting statistic over the set of $321$-avoiding permutations. Namely, we let $\beta(\sigma)$ denote the number of occurrences of the arrow pattern $\beta=(12;1\to 2)$ in $\sigma$, and derive the following Theorem~\ref{thm:beta-Kra}. 

We first recall a bijection $\PK$ due to Krattenthaler \cite[Section 4]{Kra2001} (see also \cite[Chapter 1.5]{EC1}) that sends each $321$-avoiding permutation to a Dyck path. Let $\Dyck_n$ denote the set of Dyck paths of semilength $n$. Given a permutation $\sigma\in\fS_n(321)$, let 
$$\LRMax(\sigma)=\{m_1=\sigma_1,m_2,\ldots,m_k=n\}_{<}\, .$$  
We see that $\sigma$ has a unique decomposition
$$\sigma = m_1v_1m_2v_2m_3\cdots m_{k-1}v_{k-1}m_kv_k,$$
where each subword $v_i$ is either empty or monotonically increasing (to avoid pattern $321$). The $321$-avoidance forces the concatenation $v_1v_2\cdots v_k$ to be increasing as well. Beginning initially with $w^{(0)}=\varnothing$, the empty word, for each $i=1,2,\ldots,k$ we append to $w^{(i-1)}$ $(m_i-m_{i-1})$ copies of $\rU$ (with a convention that $m_0=0$), followed by $(|v_i|+1)$ copies of $\rD$, and denote the newly obtained word by $w^{(i)}$. Here $|v_i|$ refers to the length of word $v_i$. The final word $w^{(k)}$ is taken as our image Dyck path $\PK(\sigma)$. The reader is referred to \cite{Kra2001} or \cite{EC1} for a complete proof that $\PK$ is indeed a well-defined bijection. 

If we represent permutation $\sigma$ using the associated \emph{permutation matrix} $P_{\sigma}$\footnote{the $(i,j)$-entry of $P_{\sigma}$ is $1$ if and only if $\sigma_i=j$, and we replace each $1$ by a cross ``$\mathrm{X}$'' for better illustration.}, then its image $\PK(\sigma)$ is seen to be the closest Dyck path to the diagonal that bounds all crosses to its southeast. See Fig.~\ref{Fig:K-map} below for an example of $\PK$, where the image Dyck path has been rotated counterclockwise $45\degree$ for better illustration.

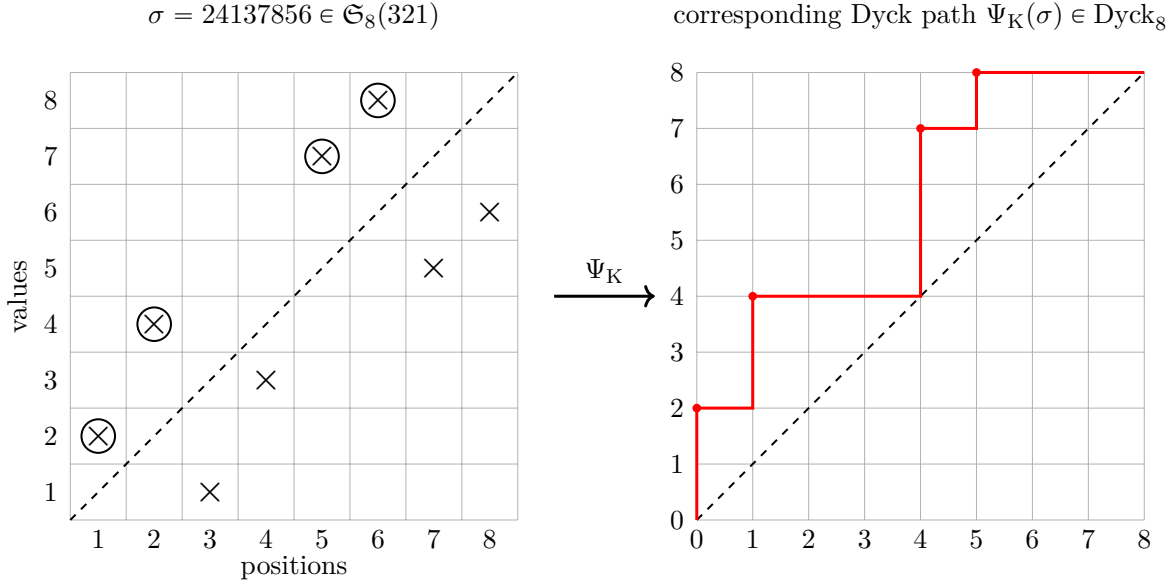
\begin{figure}[htbp]
	\centering
	\resizebox{0.96\textwidth}{!}{%
		\begin{tikzpicture}[scale=0.78, every node/.style={font=\small}]
			
			\begin{scope}[shift={(0,0)}]
				
				\node at (4,9) {$\sigma=24137856 \in \fS_8(321)$};

				\draw[step=1, gray!60] (0,0) grid (8,8);
				
				\draw[dashed, thick] (0,0) -- (8,8);
				
				\foreach \j in {1,...,8}
				\node at (\j-0.5,-0.35) {\j};
				\foreach \i in {1,...,8}
				\node at (-0.35,\i-0.5) {\i};
				
				\node at (4,-0.85) {positions};
				\node[rotate=90] at (-0.95,4) {values};
				
				\foreach \i/\j in {1/2,2/4,3/1,4/3,5/7,7/5,6/8,8/6}{
					\node at (\i-0.5,\j-0.5) {\Large $\times$};
				}
				
				\foreach \i/\j in {1/2,2/4,5/7,6/8}{
					\draw[thick] (\i-0.5,\j-0.5) circle (0.30);
				}
				
			\end{scope}
			
			\draw[->, very thick] (8.65,4) -- (10.45,4);
			\node at (9.55,4.45) {$\PK$};
			
			\begin{scope}[shift={(11.2,0)}]
				
				\node at (4,9) {corresponding Dyck path $\PK(\sigma)\in\Dyck_8$};

				\draw[step=1, gray!60] (0,0) grid (8,8);
				
				\draw[dashed, thick] (0,0) -- (8,8);
				
				\foreach \j in {0,...,8}
				\node at (\j,-0.35) {\j};
				\foreach \i in {0,...,8}
				\node at (-0.35,\i) {\i};
				
				\draw[very thick, red]
				(0,0) -- (0,2) -- (1,2) -- (1,4) -- (4,4)
				-- (4,7) -- (5,7) -- (5,8) -- (8,8);
				
				\foreach \x/\y in {0/2,1/4,4/7,5/8}{
					\fill[red] (\x,\y) circle (2.2pt);
				}
				
			\end{scope}
			
		\end{tikzpicture}%
	}
	\caption{An example of $\PK$ \label{Fig:K-map}}
\end{figure}

Given any path $w\in\Dyck_n$, let us denote by $\tu(w)$ the number of occurrences of three consecutive $\rU$'s in $w$, and let $\rev(w)$ be the Dyck path derived from reversing $w$. Taking the Dyck path $w=\rU\rU\rD\rU\rU\rD\rD\rD\rU\rU\rU\rD\rU\rD\rD\rD$ shown in Fig.~\ref{Fig:K-map} as an example, we have $\tu(w)=1$ and $\rev(w)=\rU\rU\rU\rD\rU\rD\rD\rD\rU\rU\rU\rD\rD\rU\rD\rD$. Interestingly, this statistic $\tu$ also appears in a recent work \cite{FF2026} by Fang and the first author, where it is related to certain statistics on regions of the Catalan arrangement.

\begin{theorem}\label{thm:beta-Kra}
The mapping $\Psi:=\rev\circ \PK:\fS_n(321)\to\Dyck_n$ is a bijection such that for every $\sigma\in\fS_n(321)$, we have
\begin{align}\label{stat-PK}
	\beta(\sigma) &= \tu(\Psi(\sigma)).
\end{align}
In particular,
\begin{align}\label{eq:tu=0}
	a_n(321,\beta)=|\{w\in\Dyck_n: \tu(w)=0\}|=M_n,
\end{align} 
thus Conjecture~\ref{conj:Archer-Laudone}~(3) holds true.
\end{theorem}
\begin{proof}
Since we already know that both $\PK$ and $\rev$ are bijections, it suffices to show \eqref{stat-PK} and \eqref{eq:tu=0}. Recall that a $321$-avoiding permutation $\sigma$ has a unique decomposition
$$\sigma = m_1v_1m_2v_2m_3\cdots m_{k-1}v_{k-1}m_kv_k,$$
where $m_1, m_2,\ldots, m_k$ are all of the left-to-right maxima, and the concatenation $v_1v_2\cdots v_k$ is monotonically increasing. Note that each occurrence of the arrow pattern $\beta$ corresponds to an ascent in a certain nonempty subword $v_i$. But now that each $v_i$ is increasing, it contributes $|v_i|-1$ to $\beta(\sigma)$. On the other hand, a consecutive run of $|v_i|+1$ copies of $\rD$ (with $|v_i|>0$) in $\PK(\sigma)$ corresponds to a run of $|v_i|+1$ copies of $\rU$ in $\rev(\PK(\sigma))=\Psi(\sigma)$, thereby contributing $|v_i|-1$ to $\tu(\Psi(\sigma))$. This proves \eqref{stat-PK}.

To see \eqref{eq:tu=0}, we consider the bivariate generating function $F(t,z):=\sum_{n\ge 0}z^n\sum_{w\in\Dyck_n}t^{\tu(w)}$, which satisfies the following functional equation (see \cite[A092107]{OEIS}):
\begin{align}\label{gf:tu}
	z(t+z-tz)F^2(t,z)-(1-z+tz)F(t,z)+1 = 0.
\end{align}
Setting $t=0$ in \eqref{gf:tu} results in 
$$z^2F^2(0,z)+(z-1)F(0,z)+1=0,$$
or equivalently,
$$F(0,z)-1 = zF(0,z)+z^2F^2(0,z).$$
This is the same equation satisfied by the generating function of the Motzkin numbers, so we have \eqref{eq:tu=0}.
\end{proof}

\section{Real-rootedness}\label{sec:real_rootedness}
The study of real-rootedness investigates when the generating polynomials of combinatorial sequences have exclusively real roots, a property that directly guarantees fundamental structural traits such as unimodality, log-concavity, and asymptotic normality. As highlighted in the influential survey papers by Brenti~\cite{Bre1989} and Br\"and\'en~\cite{Bra2014}, this field connects combinatorics with real algebraic geometry and analysis through techniques like total positivity, interlacing roots, and root-preserving linear operators.

The results in this section were largely motivated by our communication with Per Alexandersson; we are grateful to him for sharing his observations and kindly permitting us to include them here. 

We begin by recalling \eqref{eq: des-321-beta=Motzkin poly}:
\begin{align}\label{F-des-321-beta}
	F^{\des}_{321,\beta,n}(t):= \sum_{\sigma\in \fS_n(321,\beta)} t^{\des(\sigma)}=
	\sum_{k=0}^{\lfloor n/2\rfloor}\binom{n}{2k}C_k t^k=:M_n(t).
\end{align}
Here, $M_n(t)$ enumerates Motzkin paths of length $n$ according to their number of up steps, and is known as the $n$-th \emph{Motzkin polynomial} (also referred to as a Jacobi--Rogers polynomial)~\cite[\href{https://oeis.org/A055151}{A055151}]{OEIS}. Furthermore, $M_n(t)$ satisfies the following recurrence relation\footnote{\url{https://www.symmetricfunctions.com/realRootedCatalan.htm##realRootedCatalan}}, with initial conditions $M_1(t)=1$, $M_2(t)=t+1$, and for $n\ge 3$,
\begin{align*}
	(n+2)M_n(t) &= (2n+1)M_{n-1}(t)-(n-1)(1-4t)M_{n-2}(t).
\end{align*}
This recurrence directly implies that $M_n(t)$ is \emph{real-rooted} (i.e., all its roots are real). For instance, by applying a result of Liu and Wang~\cite[Theorem 2.1]{LW07}, one can show inductively that $\{M_n(t)\}_{n\ge 1}$ forms a Sturm sequence --- a property stronger than real-rootedness.
\begin{corollary}
The sequence of polynomials $\{F^{\des}_{321,\beta,n}(t)\}_{n\ge 1}$ forms a Sturm sequence; that is, the roots of $F^{\des}_{321,\beta,n-1}(t)$ interlace those of $F^{\des}_{321,\beta,n}(t)$. In particular, all roots of these polynomials are real.
\end{corollary}

In parallel with \eqref{F-des-321-beta}, one may also investigate the distribution of descents over $\fS_n(123,\alpha)$. To this end, define $F_n(t)=F^{\des}_{123,\alpha,n}(t):=\sum_{\sigma\in \fS_n(123,\alpha)} t^{\des(\sigma)}$. Using tools available at \url{https://poly.symmetricfunctions.com/}, Alexandersson empirically observed\footnote{\label{fn1}P.~Alexandersson, personal communication, July 7th, 2026.} the recurrence relation~\eqref{eq: rec of Alex}, where $F'_n(t)$ denotes the derivative of $F_n(t)$ with respect to $t$. This observation ultimately motivated the following two results. Note that the real-rootedness claim in Theorem~\ref{thm:RR for des-123-alpha} was previously conjectured by Alexandersson.\footnotemark[\getrefnumber{fn1}]
\begin{theorem}
Let $F_n(t)=\sum_{\sigma\in \fS_n(123,\alpha)} t^{\des(\sigma)}=\sum_{k=0}^{n-1}f_{n,k}t^k$, then we have
\begin{align}
f_{n,n-1} &= 1,\label{f: k=n-1}\\
f_{n,n-2} &= (n-1)^2,\label{f: k=n-2}\\
f_{n,k} &= 2\binom{n}{2(n-1-k)}, \text{ for all $0\le k\le n-3$. }\label{f: k<=n-3}
\end{align}
In particular, $F_1(t)=1$, $F_2(t)=t+1$, $F_3(t)=t^2+4t$, and for $n\ge 4$, 
\begin{equation}
\begin{split}
	(3-n)F_n(t) = {} & (3-n)(3-n+t)F_{n-1}(t)+((6-2n)t-2t^2)F'_{n-1}(t) \\
	& +((3-n)(2-n)t+(2n-4)t^2)F_{n-2}(t).
\end{split}
\label{eq: rec of Alex}
\end{equation}
\end{theorem}
\begin{proof}
The fact that the permutation with the most descents (namely, $n\cdots 2 1$) is unique and indeed belongs to $\fS_n(123,\alpha)$ explains \eqref{f: k=n-1}. We derive both \eqref{f: k=n-2} and \eqref{f: k<=n-3} by induction. The base cases are easily verified by examining the coefficients of $F_2(t),F_3(t)$, and $F_4(t)=t^3+9t^2+2t$.

We first prove \eqref{f: k=n-2}. Take any permutation $\sigma\in\fS_n(123,\alpha)$ and suppose that $\des(\sigma)=n-2$. If $\sigma_n=n$, then the only possibility is that $\sigma=(n-1)(n-2)\cdots 1 n$. Otherwise we assume that $\sigma_n=k$ for some $1\le k\le n-1$. Applying Lemma~\ref{lem:char of alpha}, we can assume the decomposition
\begin{align}\label{decomp for alpha}
\sigma = \tau_1\cdots\tau_i\eta'_1\cdots\eta'_{n-k}\tau_{i+2}\cdots\tau_k k,
\end{align}
for some $0\le i<k$, where $\tau:=\tau_1\cdots\tau_ik\tau_{i+2}\cdots\tau_k\in\fS_k(\alpha)$ and $\eta:=\eta_1\cdots\eta_{n-k}\in\fS_{n-k}(\alpha)$, with $\eta_j'=\eta_j+k$ for every $1\le j\le n-k$. There are two subcases.
\begin{enumerate}
	\item If $k=1$, then $\eta$ can be any permutation in $\fS_{n-1}(123,\alpha)$ with $n-3$ descents. So by induction we know there are $f_{n-1,n-3}=(n-2)^2$ such permutations.
	\item If $1<k<n$, then $\des(\sigma)=n-2$ implies that both $\tau_1\cdots\tau_i\tau_{i+2}\cdots\tau_k$ and $\eta$ are monotonously decreasing words, and that either $\tau_1\cdots\tau_i=\varnothing$ or $\tau_{i+2}\cdots\tau_k=\varnothing$. Hence we have two qualified permutations, namely 
	$$n(n-1)\cdots(k+1)(k-1)(k-2)\cdots 1k, \text{ and } (k-1)(k-2)\cdots 1 n(n-1)\cdots(k+1)k.$$
\end{enumerate}
Summing up all cases above, we deduce that
$$f_{n,n-2}=1+f_{n-1,n-3}+\sum_{k=2}^{n-1} 2=1+(n-2)^2+2(n-2)=(n-1)^2,$$
as desired.

The proof of \eqref{f: k<=n-3} also relies on the decomposition \eqref{decomp for alpha} and proceeds similarly to that of Theorem~\ref{thm:123-alpha}, though special care must be taken regarding the changes in descent numbers. We state the relevant recurrence relations without proof. Note that $f_{n,n-3}$ is handled separately: the permutations considered in Case 2 of the proof of Theorem~\ref{thm:123-alpha} (where $i>0$ in \eqref{decomp for alpha}) always have at least $n-3$ descents.
\begin{align}
	f_{n,n-3} &= f_{n-1,n-4}+\sum_{k=2}^{n-1}(f_{n-k,n-k-2}+k-2)=f_{n-1,n-4}+2\binom{n-1}{3},\label{rec:f-n-3}\\
	f_{n,j} &= f_{n-1,j-1}+\sum_{k=2}^{n-1}f_{n-k,j-k+1}, \text{ for all $0\le j\le n-4$. }\label{rec:f-j}
\end{align}
As usual, $f_{n,k}$ is understood to be $0$ for $k<0$ or $k\ge n$. It is now routine to check that $2\binom{n}{2(n-1-j)}$, with $0\le j\le n-3$, also satisfies \eqref{rec:f-n-3} and \eqref{rec:f-j}. This completes the inductive proof of \eqref{f: k<=n-3}.

Having confirmed the explicit expression of $f_{n,k}$, the verification of \eqref{eq: rec of Alex} becomes straightforward and is thus omitted.
\end{proof}

\begin{theorem}\label{thm:RR for des-123-alpha}
Let $\widetilde{F}_n(t):=t^{n-1}F_n(t^{-1})$. For all $n\ge 1$, we have
\begin{align}\label{tildeF}
\widetilde{F}_n(t) &= (1+\sqrt{t})^n+(1-\sqrt{t})^n-(n-1)t-1.
\end{align}

Consequently, the polynomial $F_n(t)$ is real-rooted for all $n\ge 1$.
\end{theorem}
\begin{proof}
Plugging in the expressions of the coefficients given by \eqref{f: k=n-1}, \eqref{f: k=n-2}, and \eqref{f: k<=n-3}, we deduce that for $n\ge 1$,
\begin{align*}
\widetilde{F}_n(t) &= 1+(n-1)^2t+2\sum_{k=2}^{\lfloor\frac{n}{2}\rfloor}\binom{n}{2k}t^k=-1+(1-n)t+\sum_{k=0}^{n}\binom{n}{k}(\sqrt{t}^k+(-\sqrt{t})^k)\\
&= (1+\sqrt{t})^n+(1-\sqrt{t})^n-(n-1)t-1,
\end{align*}
which is \eqref{tildeF}.

To show that $\widetilde{F}_n(t)$ is real-rooted (and consequently $F_n(t)$ as well), given that $\deg(\widetilde{F}_n(t))=\lfloor n/2\rfloor$, it suffices to demonstrate that $\widetilde{F}_n(t)$ possesses $\lfloor n/2\rfloor$ distinct real roots. Substituting $t=-\tan^2 \theta$ for $\theta\in[0,\pi/2)$, we define, using \eqref{tildeF},
\[
	g_n(\theta):=\widetilde{F}_n(-\tan^2\theta)=2\sec^n\theta\cos(n\theta)+(n-1)\tan^2\theta-1.
\]
We evaluate $g_n(\theta)$ at the sample points $\theta_m=\frac{m\pi}{n}$ for $m=0,1,\ldots,\lfloor\frac{n-1}{2}\rfloor$ and observe that:
\begin{itemize}
	\item For even $m$, since $\cos(n\theta_m)=1$ and $\sec\theta_m\ge 1$, we see that
	$$g(\theta_m)=2\sec^n\theta_m+(n-1)\tan^2\theta_m-1>0.$$
	\item For odd $m$, since $\cos(n\theta_m)=-1$ and $\tan^2\theta_m\ge 0$, we see that
	\begin{align*}
	g(\theta_m) &=-2\sec^n\theta_m+(n-1)\tan^2\theta_m-1=-2(1+\tan^2\theta_m)^{\frac{n}{2}}+(n-1)\tan^2\theta_m-1\\
	& \le -2(1+\frac{n}{2}\tan^2\theta_m)+(n-1)\tan^2\theta_m-1=-\tan^2\theta_m-3<0.
	\end{align*}
	\item Additionally, since $\deg(\widetilde{F}_n(t))=\lfloor\frac{n}{2}\rfloor$ and its leading coefficient is positive, the limit $\lim_{t\to-\infty}\widetilde{F}_n(t)=\lim_{\theta\to \pi/2}g_n(\theta)$ has the same sign as $(-1)^{\lfloor n/2\rfloor}$.
\end{itemize}
Thus, the signs of $g(\theta_m)$ strictly alternate as $+,-,+,-,\ldots$ for $m=0,1,\ldots,\lfloor\frac{n-1}{2}\rfloor$. Furthermore, because $(-1)^{\lfloor n/2\rfloor}$ matches the sign of $g(\lfloor\frac{n-1}{2}\rfloor\frac{\pi}{n})$ when $n$ is odd and opposes it when $n$ is even, $g_n(\theta)$ exhibits precisely $\lfloor n/2 \rfloor$ sign changes across $(0,\frac{\pi}{2})$ regardless of the parity of $n$. By the Intermediate Value Theorem, this yields $\lfloor n/2 \rfloor$ distinct real roots for $g_n(\theta)$ in this interval, establishing that $\widetilde{F}_n(t)$ --- and consequently $F_n(t)$ --- has $\lfloor n/2 \rfloor$ distinct real roots, as claimed.
\end{proof}

\section{Concluding remarks}\label{sec:concluding_remarks}

As highlighted by the work of Claesson and Kitaev~\cite{CK2008}, in the literature there are at least nine different bijections between $\fS_n(231)$ and $\fS_n(321)$. It is then natural to wonder if any of them could lead to a direct bijective proof of
$$a_n(231,\alpha)=a_n(321,\alpha).$$
Using the FindStat database~\cite{findstat}, we observe that the Simion-Schmidt bijection \cite{SS1985} (see also \cite[Section~3.3]{CK2008}) accomplishes exactly that. The details are left to the interested reader. Note that however, in contrast with \eqref{id:gf-213-231-q}, the stronger relation 
$$R^{231}(q,x)=R^{321}(q,x)$$ 
does NOT hold. For example, among permutations in $\fS_5(321)$, the only one that has two occurrences of $\alpha$ is $41523$, while in $\fS_5(231)$ both $51423$ and $51432$ have two occurrences of $\alpha$. For the reader's convenience, we include Table~\ref{tab:R^321(q,x)}, whose $(n,k)$-entry records the coefficient of $q^kx^n$ in $R^{321}(q,x)$ --- equivalently, the number of $321$-avoiding permutations of length $n$ with exactly $k$ occurrences of the arrow pattern $\alpha$.

\smallskip 
\begin{table}[!h]
\renewcommand{\arraystretch}{1.1}
    \centering
    \begin{tabular}{c|ccccccccccc}
        $n \backslash k$ & 0 & 1 & 2 & 3 & 4 & 5 & 6 & 7 & 8 & 9 \\ \hline
        1 & 1 \\ 
        2 & 2 \\ 
        3 & 5 \\ 
        4 & 13 & 1  \\ 
        5 & 34 & 7 & 1 \\ 
        6 & 89 & 32 & 9 & 2 \\
        7 & 233 & 122 & 50 & 20 & 3 & 1 \\
        8 & 610 & 422 & 223 & 121 & 35 & 15 & 3 & 1 \\
        9 & 1597 & 1376 & 879 & 579 & 240 & 124 & 43 & 18 & 4 & 2
    \end{tabular}
    \vspace{3mm}
    \caption{The distribution of the $\alpha$-count over $321$-avoiding permutations \label{tab:R^321(q,x)} }
\end{table}

In the same vein but with the roles of the classical and arrow patterns interchanged, we could consider the distributions of the number of occurrences of classical pattern $132$ or $312$ over $\fS_n(\alpha)$, i.e., the generating functions
\begin{align*}
	S^{132}(q,x) &:= 1+\sum_{n\ge 1}x^n\sum_{\sigma\in\fS_n(\alpha)}q^{132(\sigma)},\\
	S^{312}(q,x) &:= 1+\sum_{n\ge 1}x^n\sum_{\sigma\in\fS_n(\alpha)}q^{312(\sigma)}.
\end{align*}
Then by \cite[Theorem 3.5]{AL2026} and Theorem~\ref{thm:132-alpha}, we see that 
\begin{align*}
&[x^n]S^{132}(1,x)=[x^n]S^{312}(1,x)=|\fS_n(\alpha)|=S_{n-1},\\
&[x^n]S^{132}(0,x)=a_n(132,\alpha)=[x^n]S^{312}(0,x)=a_n(312,\alpha)=C_n.
\end{align*}
Hence both polynomials $s^{132}_n(q):=[x^n]S^{132}(q,x)$ and $s^{312}_n(q):=[x^n]S^{312}(q,x)$ interpolate between the large Schr\"oder number $S_{n-1}$ and the Catalan number $C_n$. It might be of independent interest to calculate $S^{132}(q,x)$, $S^{312}(q,x)$, and investigate the Tables~\ref{tab:S^132(q,x)} and \ref{tab:S^312(q,x)} below that consist of the coefficients of $\{s^{132}_n(q)\}_{n\ge 1}$ and $\{s^{312}_n(q)\}_{n\ge 1}$.

\smallskip 
\begin{table}[!h]
\renewcommand{\arraystretch}{1.1}
    \centering
    \begin{tabular}{c|cccccccccccccc}
        $n \backslash k$ & 0 & 1 & 2 & 3 & 4 & 5 & 6 & 7 & 8 & 9 & 10 & 11 & 12 \\ \hline
        1 & 1 \\ 
        2 & 2 \\ 
        3 & 5 & 1  \\ 
        4 & 14 & 3 & 4 & 1  \\ 
        5 & 42 & 9 & 13 & 10 & 9 & 4 & 3 \\ 
        6 & 132 & 28 & 41 & 32 & 48 & 18 & 37 & 17 & 17 & 13 & 9 & 0 & 2 \\
    \end{tabular}
    \vspace{3mm}
    \caption{The distribution of the $132$-count over $\alpha$-avoiding permutations \label{tab:S^132(q,x)} }
\end{table}

\begin{table}[!h]
\renewcommand{\arraystretch}{1.1}
    \centering
    \begin{tabular}{c|cccccccccccccc}
        $n \backslash k$ & 0 & 1 & 2 & 3 & 4 & 5 & 6 & 7 & 8 & 9 & 10 & 11 & 12 \\ \hline
        1 & 1 \\ 
        2 & 2 \\ 
        3 & 5 & 1  \\ 
        4 & 14 & 4 & 3 & 1  \\ 
        5 & 42 & 15 & 13 & 10 & 5 & 2 & 3 \\ 
        6 & 132 & 56 & 53 & 41 & 38 & 17 & 26 & 9 & 5 & 10 & 5 & 0 & 2 \\
    \end{tabular}
    \vspace{3mm}
    \caption{The distribution of the $312$-count over $\alpha$-avoiding permutations \label{tab:S^312(q,x)} }
\end{table}

\section*{Acknowledgement} 
Shishuo Fu was supported by the National Natural Science Foundation of China (grants no. 12171059 and 12671383) and the Fundamental Research Funds for the Central Universities (grant no.~2025CDJ-IAISYB-008).

\end{document}